\documentclass[11pt, epsfig]{article}
\usepackage{epsfig, amsmath, amssymb, amsthm, times}
\usepackage[T1]{fontenc}
\usepackage{hyperref}

\def\<{\langle}
\def\>{\rangle}

\def\i{\underline i}

\def\a{\underline a}
\def\c{\underline c}
\def\b{\underline b}
\def\x{\underline x}
\def\z{\underline z}
\def\w{\underline w}
\def\v{\underline v}

\newcommand{\be}{\begin{equation}}
\newcommand{\ee}{\end{equation}}
      \newtheorem{theorem}{Theorem}[section]
       \newtheorem{proposition}[theorem]{Proposition}

       \newtheorem{remark}{Remark}[section]
\theoremstyle{definition}

\title{An eigenproblem approach  to \\optimal equal-precision sample allocation in subpopulations}
\author{Jacek Weso\l owski\thanks{Politechnika Warszawska and G\l \'owny Urz\k{a}d Statystyczny, Warszawa, Poland, e-mail address: wesolo@mini.pw.edu.pl},\hspace{2mm} Robert Wieczorkowski\thanks{G\l \'owny Urz\k{a}d Statystyczny, Warszawa, Poland, e-mail address: R.Wieczorkowski@stat.gov.pl} }
\begin{document}
\maketitle
\begin{abstract}
Allocation of samples in stratified and/or multistage sampling is one of the central issues of sampling theory.
In a survey of a population often the constraints for precision of estimators of subpopulations parameters have to be taken care of during the allocation of the sample. Such issues are often solved with mathematical programming procedures. In many situations it is desirable to allocate the sample, in a way which forces the precision of estimates at the subpopulations level to be both: optimal and identical, while the constraints of the total (expected) size of the sample (or samples, in two-stage sampling) are imposed. Here our main concern is related to two-stage sampling schemes. We show that such problem in a wide class of sampling plans has an elegant mathematical and computational solution. This is done due to a suitable definition of the optimization problem, which enables to solve it through a linear algebra setting involving eigenvalues and eigenvectors of matrices defined in terms of some population quantities. As a final result we obtain a very simple and relatively universal method for calculating the subpopulation optimal and equal-precision allocation which is based on one of the most standard algorithms of linear algebra (available e.g. in {\em R} software). Theoretical solutions are illustrated through a numerical example based on the Labour Force Survey. Finally, we would like to stress that the method we describe, allows to accommodate quite automatically for different levels of precision priority for subpopulations.
\end{abstract}
\section{Introduction}
Consider a population $U$ partitioned into subpopulations $U_1,\ldots,U_J$, that is $U=\bigcup_{j=1}^J\,U_i$ and $U_i\cap U_j=\emptyset$ for $i\ne j$, $i,j=1,\ldots,J$. Assume that we are interested in estimation of means of a variable $\mathcal{Y}$ in all subpopulations. In each $U_i$ a sample of $n_i$ elements, $i=1,\ldots,J$, is chosen according to simple random sampling without replacement (SRSWOR). Assume additionally that the size of the total sample $n=n_1+\ldots+n_J$ is fixed. A natural requirement is to allocate the sample among subpopulations in such a way that the precision (here and throughout the paper understood as coefficients of variation, CV's) of the estimators in each of the subpopulations are the same. Throughout this paper by equal-precision we always mean equal-precision in subpopulations. That is we want to have
$$
T_j=\left(\tfrac{1}{n_j}-\tfrac{1}{N_j}\right)\gamma_j^2=\mathrm{const}=:T\qquad \mbox{for all }\;j=1,\ldots,J,
$$
where $\gamma_j$ is the CV of $\mathcal{Y}$ in $U_j$ and $N_j=\#\,U_j$, $j=1,\ldots,J$.
Expressing $n_j$ in terms of $T_j$, $j=1,\ldots,J$, the constraint on the size of the total sample $n=n_1+\ldots+n_J$ gives the equation
$$
n=\sum_{j=1}^J\,\tfrac{N_j\gamma_j^2}{\gamma_j^2+TN_j}
$$
with unknown $T$. It is easy to see that the above equation has a unique solution, which can be easily computed numerically (however no analytical explicit formula is available). Obviously, such a solution, $T^*$, gives the optimal allocation $n_j=\tfrac{N_j\gamma_j^2}{\gamma_j^2+T^*N_j}$, $j=1,\ldots,J$ -  for more details see, e.g. Lednicki and Weso\l owski (1994).

On the other hand if one imposes requirements on CV's of estimators in subpopulations, that is when $T_j$, $j=1,\ldots,J$, are given (not necessarily identical) there is no freedom in the sense that they determine uniquely the total sample size. If instead one assumes only the restriction that CV's of domain mean estimators are bounded from above by (possibly) different constraints $T_j$, $j=1,\ldots,J$, the minimization of the total sample size is a valid question. It has been solved recently (with additional constraint on the CV of the estimator of the population mean) for stratified SRSWOR by Choudhry, Rao and Hidiroglou (2012) through a nonlinear programming (Newton-Raphson) procedure. These authors followed earlier application of such procedures to optimal allocation of the sample among strata for the population means estimation in multivariate setting as proposed in Huddlestone, Claypool and Hocking (1970) and Bethel (1989) (see also Ch. 12.7 in S\"arndal, Swensson and Wretman (1992)).  For an alternative numerical method (Nelder-Mead simplex method) used to sample allocation (and strata construction) under multivariate setting, where subpopulations were also taken under account -  see Lednicki and Wieczorkowski (2003).

In an allocation problem for stratified two-stage sampling, when only the optimality of the estimator for the population as a whole is considered, traditionally a single constraint based on the expected total cost is imposed - see e.g. Ch. 2.8 in S\"arndal, Swensson and Wretman (1992) or Ch. 10.9 and Ch. 10.10 in Cochran (1977). Such issues were also considered more recently - see e.g. Clark and Steel (2000), Khan et al (2006) and Clark (2009) and references therein. From the practical point of view, the total cost of the two-stage survey may be difficult to model, therefore alternatively, constraints in terms of (expected) total sizes of SSU's (secondary sampling units) and PSU's (primary sampling units) may be imposed. Under such constraints, we are interested in the allocation of PSU's and SSU's which guarantees {\em optimal and equally precise estimators of means in all the domains}. We allow stratification on both stages and propose quite general approach to the problem which is valid not only for simple random sampling without replacement in strata. Suitable definition of the minimization issue allows to reduce it to an eigenproblem for a rank-two perturbation of a diagonal matrix. Such an approach, in the context of allocation of samples, was for the first time proposed in Niemiro and Weso\l owski (2001) (denoted NW in the sequel), where stratification was allowed only either at the first or at the second stage with SRSWOR on both stages only. Moreover, some technical conditions were required in that paper, which allowed to use the famous Perron-Frobenius theorem in the proof of the main result. Such  approach was applied in agricultural surveys in Kozak (2004) and in a forestry survey in Kozak and Zieli\'nski (2005). Also it has been slightly developed theoretically by allowing  CV's in domains to be of the form $\kappa_j T$, $j=1,\ldots,J$, with known coefficients (levels of priority) $\kappa_j$, $j=1,\ldots,J$, and unknown optimal $T$, in Kozak, Zieli\'nski and Singh (2008).

The eigenvalue-eigenvector solution we present here is quite universal. In particular, it covers also classical Neyman optimal allocation as a boundary case of $J=1$, that is the case of one subpopulation (see, e.g. S\"arndal, Swensson, Wretman (1992), Ch. 3.7.3), or related solution in the two- stage sampling (see e.g. S\"arndal, Swensson, Wretman (1992), Ch. 12.8.1). It includes also equal-precision allocation derived in Lednicki and Weso\l owski (1994). Though these issues are rather standard and well-understood, as far as we know, they have never been embedded into and eigenproblem setting. Nevertheless, our main concern here is two-stage sampling schemes. The point of departure for the present paper is that of NW, where optimal equal-precision allocation was considered in two special cases of two-stage sampling: (1) stratified SRSWOR at the first stage and SRSWOR at the second, (2) SRSWOR at the first stage and stratified SRSWOR at the second. In that paper eigenproblem approach together with the Perron-Frobenius theory of positive matrices were used. Here we develop a similar approach (though we go beyond the Perron-Frobenius theory) to a wider class of two-stage sampling schemes and with less restrictive requirements for the population characteristics. It is given in Section 3, where an eigenvalue-eigenvector solution of a general minimization problem leads to optimal and equal precision estimators in subpopulations for some stratified sampling plans. Section 2 is a kind of a warm-up: the proposed method is introduced in rather standard settings of single-stage sampling. A numerical example comparing the eigenvalue-eigenvector allocation with the standard one in the Polish Labour Force Survey is discussed in Section 4. Conclusions, involving incorporation of different levels of priority in the proposed method, are given in Section 5.

\section{Equal-precision optimal allocation in single-stage sampling}
In this section we consider single-stage stratified simple random sampling. The main purpose of this section is to give a friendly introduction to the approach via eigenvalues and eigenvectors, since in this case the proofs are less complicated than in the case of two-stage sampling.  It also confirms relative universality of such purely linear-algebraic solution to the allocation problem, when the constraint of equal-precision is imposed. Nevertheless, in a single-stage setting one can use alternatively a direct numerical method as described in Remarks \ref{LW} or a combination of the direct numerical approach and the Neyman optimal allocation method as explained in \ref{stra} below. We would like to emphasize that such approaches are not possible in the two-stage setting considered in Section 3.

We start with a general minimization problem, which, we show, can be treated through linear algebra methods.
\subsection{Minimization problem - generalities}
Consider strictly positive numbers: $c_j$, $A_{j,h}$, $h=1,\ldots,H_j$, $j=1,\ldots,d$ and $x$.
Denote \be\label{ac}\a=(a_j,\,j=1,\ldots,d)=\tfrac{1}{\sqrt{x}}\,\left(\sum_{h=1}^{H_j}\,\sqrt{A_{j,h}},\,j=1,\ldots,d\right),\qquad \c=(c_j,\,j=1,\ldots,d).\ee Let
\be\label{matac}
D=\a\a^T-\mathrm{diag}(\c).
\ee

\begin{theorem}\label{ss}
Assume that $D$ as defined in \eqref{matac}, has the unique simple positive eigenvalue $\lambda$ and let $\v=(v_1,\ldots,v_d)$ be a respective eigenvector. Then the problem of minimization of
\be\label{tmin2}
T=\sum_{h=1}^{H_j}\,\tfrac{A_{j,h}}{x_{j,h}}-c_j,\quad j=1,\ldots,d,
\ee
where $x_{j,h}>0$, $h=1,\ldots,H_j$, $j=1,\ldots,d$, under the constraint
\be\label{wiez21}
\sum_{j=1}^d\,\sum_{h=1}^{H_j}\,x_{j,h}=x,
\ee
where $x$ is a given positive number, has the solution
$$
x_{j,h}=x\,\tfrac{v_j\sqrt{A_{j,h}}}{\sum_{k=1}^d\,v_k\,\sum_{g=1}^{H_k}\,\sqrt{A_{k,g}}}.
$$

Moreover, $T=\lambda$, the unique positive eigenvalue of matrix $D$.
\end{theorem}
\begin{proof}
For $\x=(x_{j,h},\,h=1,\ldots,H_j,\;j=1,\ldots,d)$ consider the Lagrange function
$$
F(T,\,\x)=T+\sum_{j=1}^d\,\mu_j\left(\sum_{h=1}^{H_j}\,\tfrac{A_{j,h}}{x_{j,h}}-c_j-T\right)+\mu\,\sum_{j=1}^d\,\sum_{h=1}^{H_j}\,x_{j,h}.
$$

Differentiate with respect to $x_{j,h}$ to get equations for stationary points
$$
\tfrac{\partial\,F}{\partial\,x_{j,h}}=\mu-\mu_j\tfrac{A_{j,h}}{x_{j,h}^2}=0.
$$
Therefore $\mu/\mu_j>0$ and
$$
x_{j,h}=\sqrt{\tfrac{\mu_j}{\mu}}\,\sqrt{A_{j,h}},\qquad h=1,\ldots,H_j,\;j=1,\ldots,d.
$$
Plugging it to \eqref{wiez21} we obtain
\be\label{xsum}
x=\sum_{j=1}^d\,v_j\,\sum_{h=1}^{H_j}\,\sqrt{A_{j,h}},
\ee
where $v_j=\sqrt{\mu_j/\mu}$, $j=1,\ldots,d$. Now the constraint \eqref{tmin2} gives
$$\sum_{h=1}^{H_j}\,\sqrt{A_{j,h}}-c_jv_j=Tv_j,\qquad j=1,\ldots,d.$$
By \eqref{xsum} it can be written as
$$
\tfrac{1}{x}\left(\sum_{k=1}^d\,v_k\sum_{g=1}^{H_k}\,\sqrt{A_{k,g}}\right)\,\sum_{h=1}^{H_j}\,\sqrt{A_{j,h}}-c_jv_j=Tv_j,\qquad j=1,\ldots,d.
$$
Alternatively, it can be written as
$$
D\,\v=T\,\v,
$$
where the matrix $D$ is defined in \eqref{matac}. That is $0<T=\lambda$ is the unique positive eigenvalue and $\v$ is the eigenvector related to $T$.

To show that the eigenvector $\v$ attached to the eigenvalue $\lambda$ has all coordinates of the same sign we use the celebrated Perron-Frobenius theorem: {\em If $A$ is a matrix with all strictly positive entries then there exists a simple positive eigenvalue $\mu$ such that $\mu\ge |\nu|$ for any other eigenvalue $\nu$ of $A$. The respective eigenvector (attached to $\mu$) has all entries strictly positive (up to scalar multiplication)}  - see e.g. Kato (1981), Th. 7.3 in Ch. 1.

Fix a number $\alpha>\max_{1\le j\le J}\,c_j>0$. The matrix $D+\alpha I$, where $I$ is the identity matrix, has all entries strictly positive. For any eigenvalue $\delta_j$ of $D$ and respective eigenvector $\w_j$
$$
(D+\alpha I)\w_j=(\delta_j+\alpha)\w_j,\qquad j=1,\ldots,d.
$$
That is $\delta_j+\alpha$, and $\w_j$, $j=1,\ldots,d$, are respective eigenvalues and eigenvectors of the matrix $D+\alpha I$. By the Perron-Frobenius theorem, there exists $j_0$ such that $\delta_{j_0}+\alpha\ge |\delta_j+\alpha|$ for any $j$ and respective eigenvector $\w_{j_0}$ has all entries of the same sign. Consequently, $\delta_{j_0}+\alpha\ge \delta_j+\alpha$, and thus $\delta_{j_0}\ge \delta_j$ for any $j$. Therefore, by assumption that $\lambda$ is the unique simple positive eigenvalue of $D$ it follows that $T=\lambda=\delta_{j_0}$ and the respective eigenvector $\v=\w_{j_0}$ has all entries of the same sign.
\end{proof}

\begin{proposition}\label{linalg}
Let $\a,\,\c\in(0,\infty)^d$ be such that
\be\label{zal}
\sum_{i=1}^d\,\tfrac{a_i^2}{c_i}>1.
\ee

Then the matrix $D$ defined in \eqref{matac}
has a unique simple positive eigenvalue $\lambda$.
\end{proposition}
\begin{proof}
For any $d\times d$ Hermitian matrix $M$ denote by $\lambda_i(M)$, $i=1,\ldots,d$, non-decreasingly ordered eigenvalues of $M$.
Recall the Weyl inequalities (see e.g. Th. 4.3.1 in Horn and Johnson (1985)): {\em Let $A$ and $B$ be Hermitian $d\times d$ matrices. Then
\be\label{weyl1}
\lambda_k(A+B)\le \lambda_k(A)+\lambda_d(B)\qquad \forall\,k=1,\ldots,d.
\ee}

Note first that $D$, as rank one perturbation of diagonal matrix, is non-singular and thus all eigenvalues of $D$ are non-zero.
Since $\a\a^T$ is non-negative definite of rank 1, we have $\lambda_j(\a\a^T)=0$ for $j=1,\ldots,d-1$. Thus \eqref{weyl1} with $A=\a\a^T$ and $B=-\mathrm{diag}(\c)$ for $k=d-1$ implies $\lambda_{d-1}(D)< 0$ since all eigenvalues of $B$ are negative and all eigenvalues of $D$ are non-zero.

Therefore,
\be\label{detd}
\mathrm{sgn}\,\det\,D=(-1)^{d-1}\mathrm{sgn}(\lambda_d(D)).
\ee

On the other hand expanding determinant of $D$
$$
\det\,D=\det\,\left[\begin{array}{cccc}
a_1^2-c_1 & a_1a_2 & \ldots & a_1a_d \\
a_1a_2 & a_2^2-c_2 & \ldots & a_2a_d \\
\ldots & \ldots & \ldots & \ldots \\
a_1a_d & a_2a_d & \ldots & a_d^2-c_d
\end{array} \right]
$$
and using the fact that $\a\a^T$ is of rank one we obtain
$$
\det\,D=(-1)^{d-1}\,\sum_{i=1}^d\,a_i^2\prod_{\substack{k=1\\k\ne i}}^d\,c_k+(-1)^d\prod_{k=1}^d\,c_k=(-1)^d\prod_{k=1}^d\,c_k\,\left(1-\sum_{i=1}^d\,\tfrac{a_i^2}{c_i}\right).
$$

Comparing the above formula with \eqref{detd} we have
$$
\mathrm{sgn}(\lambda_d(D))=-\mathrm{sgn}\left(1-\sum_{i=1}^d\,\tfrac{a_i^2}{c_i}\right).
$$
Due to \eqref{warun} we get $\lambda:=\lambda_d(D)>0$.
\end{proof}

\subsection{Application to stratified simple random sampling}
The population $U=\{1,\ldots,N\}$ consists of subpopulations $U_1,\ldots,U_J$, that is $U=\bigcup_{j=1}^J\,U_j$ and $U_i\cap U_j=\emptyset$ whenever $i\ne j$, $i,j=1,\ldots,J$. Consider a non-negative variable $\mathcal{Y}$ in this population, that is let $y_k=\mathcal{Y}(k)$, $k\in U$. We want to estimate the total value of $\mathcal{Y}$ in each of subpopulations, that is we are interested in $t_j=\sum_{k\in U_j}\,y_k$, $j=1,\ldots,J$. We use stratified simple random sampling without replacement (SSRSWOR) in each subpopulation. That is $U_j=\bigcup_{h=1}^{H_j}\,U_{j,h}$, $U_{j,h}\cap U_{j,g}=\emptyset$ for any $h\ne g$, $g,h=1,\ldots,H_j$, where $H_j$ is the number of strata in $U_j$, $j=1,\ldots,J$. Thus the standard estimator has the form
$$\hat{t}_j=\sum_{h=1}^{H_j}\,\tfrac{N_{j,h}}{n_{j,h}}\,\sum_{k\in s_{j,h}}\,y_k,$$
where $N_{j,h}=\#(U_{j,h})$, $n_{j,h}=\#(s_{j,h})$ and $s_{j,h}$ denotes the sample chosen from $U_{j,h}$, $h=1,\ldots,H_j$, $j=1,\ldots,J$.

Recall that its variance is
$$
D^2(\hat{t}_j)=\sum_{h=1}^{H_j}\,N_{j,h}^2\left(\tfrac{1}{n_{j,h}}-\tfrac{1}{N_{j,h}}\right)\,S_{j,h}^2,
$$
where $S_{j,h}^2$ is the population variance in $U_{j,h}$, that is $S_{j,h}^2=\tfrac{1}{N_{j,h}-1}\sum_{k\in U_{j,h}}\,(y_k-\bar{y}_{U_{j,h}})^2$ and $\bar{y}_{U_{j,h}}=\tfrac{1}{N_{j,h}}\sum_{k\in U_{j,h}}\,y_k$.

The problem we study is to allocate the sample of size $n$ among subpopulations and strata in such a way that precision (expressed through CV) of the estimation is the same and the best possible in all subpopulations. That is we would like to find the two-way array $(n_{j,h})_{\substack{h=1,\ldots,H_j\\ j=1,\ldots,J}}$ such that
\be\label{size}
\sum_{j=1}^J\,\sum_{h=1}^{H_j}\,n_{j,h}=n
\ee
and
\be\label{fix}
\tfrac{1}{t_j^2}\sum_{h=1}^{H_j}\,\tfrac{N_{j,h}^2S_{j,h}^2}{n_{j,h}}-\tfrac{1}{t_j^2}\sum_{h=1}^{H_j}\,N_{j,h}S_{j,h}^2=T,\qquad j=1,\ldots,J,
\ee
with minimal possible $T$, which, actually, is the square of the CV. Then the double array $(n_{j,h})$ is called the optimal equal-precision allocation in strata.

Define
$$
A_{j,h}:=\tfrac{N_{j,h}^2S_{j,h}^2}{t_j^2},\qquad c_j:=\sum_{h=1}^{H_j}\,\tfrac{N_{j,h}S_{j,h}^2}{t_j^2},\qquad x_{j,h}:=n_{j,h},\qquad x:=n.
$$

For such $A_{j,h}$ and $c_j$ define $\a$ and $\c$ as in \eqref{ac} and note that $d=J$ in the present setting. Let $D$ be as defined in \eqref{matac} for such $\a$ and $\c$. Then directly from Theorem \ref{ss} we obtain an allocation which ensures both optimal and the same CV's (precisions) of estimators of means  in subpopulations:

\begin{theorem}\label{eig}
Assume that the matrix $D$, as defined above, has a unique positive eigenvalue $\lambda$.

Then the  optimal equal-precision allocation under the constraint \eqref{size} is
$$
n_{j,h}=n \tfrac{v_jN_{j,h}S_{j,h}/t_j}{\sum_{k=1}^J\,v_k\,\left(\sum_{g=1}^{H_k}\,N_{k,g}S_{k,g}\right)/t_k},\qquad h=1,\ldots,H_j,\;j=1,\ldots,J,
$$
where $\v=(v_1,\ldots,v_J)^T$ is the eigenvector of $D$ associated to $\lambda$ (with all coordinates of the same sign).

Moreover, the common optimal value of the square of CV's  $T=\lambda$.
\end{theorem}

\begin{remark}\label{simpli}
{\em Assume that the overall sample size $n$ in \eqref{size} satisfies
\be\label{rest}
n<\sum_{j=1}^J\,\frac{\left(\sum_{h=1}^{H_j}\,N_{j,h}\,S_{j,h}\right)^2}{\sum_{h=1}^{H_j}\,N_{j,h}\,S_{j,h}^2}.
\ee
Let $D$ be the $J\times J$ matrix defined through \eqref{matac} with $\a$ and $\c$ as above. Then by Proposition \ref{linalg} it follows that assumptions of Theorem \ref{eig} are satisfied. Therefore the allocation using the eigenvector as given in the thesis of Theorem \ref{eig} is correct.}
\end{remark}

\begin{remark}\label{LW} {\em Consider  SRSWOR in each of subpopulations and the question of allocation $(n_1,\ldots,n_J)$ under the constraint $\sum_{j=1}^J\,n_j=n$, where $n_j$ denotes the size of the sample in $U_j$, $j=1,\ldots,J$. Moreover,
\be\label{rest1}
\tfrac{N_j^2S_j^2}{t_i^2n_j}-\tfrac{1}{t_j^2}\,N_jS_j^2=T,\qquad j=1,\ldots,J,
\ee
where $N_j=\#(U_j)$ and $S_j^2=\tfrac{1}{N_j-1}\sum_{k\in U_j}\,(y_k-\bar{y}_{U_j})^2$ with $\bar{y}_{U_j}=\tfrac{1}{N_j}\sum_{k\in U_j}\,y_k$, $j=1,\ldots,J$.

This situation is embedded in the one we considered in Theorem \ref{eig} by taking $H_j=1$, $j=1,\ldots,J$. Then the vectors $\a$ and $\c$  od Corollary \ref{linalg} are of the form
$$
\a=\tfrac{1}{\sqrt{n}}\,\left(\tfrac{1}{t_1}\,N_1S_1,\ldots,\tfrac{1}{t_N}\,N_JS_J\right)^T\qquad\mbox{and}\qquad \c=\left(\tfrac{1}{t_1^2}\,N_1S_1^2,\ldots,\tfrac{1}{t_N^2}\,N_JS_J^2\right)^T.
$$
Note that the assumption \eqref{rest} is automatically satisfied since its right hand side equals $N\;(>n)$.

With $\v$ denoting the eigenvector from Theorem \ref{eig} we obtain
$$
n_j=n\tfrac{v_jN_jS_j/t_j}{\sum_{i=1}^J\,v_iN_iS_i/t_i},\qquad j=1,\ldots,J.
$$

Alternatively, we can follow the approach described in Introduction: $T$ can be obtained as the unique solution of equation
\be\label{equa}
n=\sum_{j=1}^J\,\tfrac{N_j^2S_j^2}{T\,t_j^2+N_jS_j^2};
\ee
then
\be\label{nj1}
n_j=\tfrac{N_j^2S_j^2/t_j^2}{T+N_jS_j^2/t_j^2},\qquad j=1,\ldots,J.
\ee
}
\end{remark}

\begin{remark}\label{stra}
{\em An alternative approach to the general situation in which sub-populations are stratified is via the optimal Neyman allocation in each of sub-populations. That is we assume that for any $i=1,\ldots,J$
\be\label{njh}
n_{j,h}=n_j\tfrac{N_{j,h}S_{j,h}}{\sum_{g=1}^{H_j}\,N_{j,g}S_{j,g}},\qquad h=1,\ldots,L_j
\ee
were $n_1+\ldots+n_J=n$. Therefore, \eqref{fix} leads to
\be\label{nj}
\tfrac{\left(\sum_{h=1}^{H_j}\,N_{j,h}S_{j,h}\right)^2}{n_jt_j^2}-\tfrac{\sum_{h=1}^{H_j}\,N_{j,h}S_{j,h}^2}{t_j^2}=T,\qquad j=1,\ldots,J.
\ee
Therefore (similarly as in Introduction) we can solve \eqref{nj} for $n_j$, $j=1,\ldots,J$. Then using the constraint for the overall size of the sample we arrive at the equation for unknown $T$
\be\label{equa_st}
n=\sum_{j=1}^J\,\tfrac{\left(\sum_{h=1}^{H_j}\,N_{j,h}S_{j,h}\right)^2}{T\,t_j^2+\sum_{h=1}^{H_j}\,N_{j,h}S_{j,h}^2}.
\ee
Note that under condition \eqref{rest} a unique solution for $T$ exists. It is obtained numerically. Then  $n_j$ is obtained from \eqref{nj} and finally $n_{j,h}$ can be computed from \eqref{njh}.
One can also derive the equation \eqref{equa_st} by minimizing the sample size in each subpopulation
$n_j=n_{j,1}+\ldots+n_{j,H_j},\ j=1,\ldots,J,$ subject to a common precision $T$ (see S\"arndal, Swensson, Wretman (1992), Ch. 3.7.3).}
\end{remark}

\section{Equal-precision optimal allocation in subpopulations in two-stage sampling}
In this section we consider optimal equal-precision allocations under two-stage sampling. In the case of stratification on both stages and stratified simple random sampling without replacement (SRSWOR) we improve the result from NW in two directions. First, we relax some technical assumptions which originally were designed in order to use directly the Perron-Frobenius theorem  on eigenvalues of matrices with positive entries. Second, we allow more flexible sampling designs, as stratified  SRSWOR on both stages or Hartley and Rao (1962) systematic  $\pi$ps scheme at the first stage and SRSWOR at the second. In NW stratification was allowed either at the first or at the second stage, but not at both. Additionally, we consider a particular case of fixed sizes of samples of SSUs within PSUs - such an additional restriction is sometimes imposed in real surveys, see e.g. \L yso\'n et al. (2013), p. 28-30.

Similarly, as in the previous section we start with a more general minimization problem.

\subsection{General minimization problem}
In this section we consider and solve a minimization problem which unifies a wide class of optimal allocation problems with the same precision in subpopulations. The approach is similar to the previous section, however we have two vectors to allocate: one responsible for allocation of PSUs and one for allocation of SSUs. The direct numerical method as described in Remark \ref{LW} does not work in such two-stage setting. The reason is that there is no way to express the elements of allocation vectors in terms of unknown common precision $T$. Consequently, analogues of \eqref{equa} and \eqref{nj1} or \eqref{nj} and \eqref{equa_st} are no longer available.

Consider real numbers: $c_j>0$, $A_{j,h}>0$, $B_{j,h,i}\ge 0$, $\alpha_{j,h,i}>0$, $i\in \mathcal{V}_{j,h}$, $h=1,\ldots,H_j$, $j=1,\ldots,d$, $x,z>0$.
Denote \be\label{aa}\a=(a_j,\,j=1,\ldots,d)=\tfrac{1}{\sqrt{x}}\,\left(\sum_{h=1}^{H_j}\,\sqrt{A_{j,h}},\,j=1,\ldots,d\right),\ee \be\label{bc}\b=(b_j,\,j=1,\ldots,J)=\tfrac{1}{\sqrt{z}}\,\left(\sum_{h=1}^{H_j}\,\sum_{i\in\mathcal{V}_{j,h}}\,\sqrt{\alpha_{j,h,i}\,B_{j,h,i}},
\,j=1,\ldots,d\right),\qquad \c=(c_j,\,j=1,\ldots,d).\ee Let
\be\label{dmat}
D=\a\a^T+\b\b^T-\mathrm{diag}(\c).
\ee
\begin{theorem}\label{ogol}
Assume that $D$ has the unique positive eigenvalue $\lambda>0$ and let $\v=(v_1,\ldots,v_d)$ be a respective eigenvector. Then the problem of minimization of
\be\label{tmin}
0<T=\sum_{h=1}^{H_j}\,\tfrac{1}{x_{j,h}}\left(A_{j,h}+\sum_{i\in\mathcal{V}_{j,h}}\,\tfrac{B_{j,h,i}}{z_{j,h,i}}\right)-c_j,\quad j=1,\ldots,d,
\ee
where $x_{j,h}>0$, $z_{j,h,i}>0$, $i\in\mathcal{V}_{j,h}$, $h=1,\ldots,H_j$, $j=1,\ldots,d$, under the constraints
\be\label{wiez1}
\sum_{j=1}^d\,\sum_{h=1}^{H_j}\,x_{j,h}=x
\ee
and
\begin{equation}\label{wiez2}
\sum_{j=1}^d\,\sum_{h=1}^{H_j}\,x_{j,h}\,\sum_{i\in\mathcal{V}_{j,h}}\,\alpha_{j,h,i}\,z_{j,h,i}=z,
\end{equation}
where $x$ and $z$ are given positive numbers, has the solution
\be\label{xjjh}
x_{j,h}=x\,\tfrac{v_j\sqrt{A_{j,h}}}{\sum_{k=1}^d\,v_k\,\sum_{g=1}^{H_k}\,\sqrt{A_{k,g}}},
\ee
and
\be\label{zjjh}
z_{j,h,i}=\tfrac{z}{x_{j,h}}\;\tfrac{v_j\sqrt{\tfrac{B_{j,h,i}}{\alpha_{j,h,i}}}}{\sum_{k=1}^d\,v_k\,\sum_{g=1}^{H_k}\,\sum_{l\in\mathcal{V}_{k,g}}\,
\sqrt{\alpha_{k,g,l}\,B_{k,g,l}}}.
\ee

Moreover, $T=\lambda$, the unique positive eigenvalue of matrix $D$.
\end{theorem}
\begin{proof} The proof adapts the argument used in the proof of Theorem \ref{ss} to the more complex situation of Theorem \ref{ogol}.

Consider the Lagrange function
$$
F(T,\,\x,\,\z)=T+\sum_{j=1}^d\,\lambda_j\left(\sum_{h=1}^{H_j}\,\tfrac{1}{x_{j,h}}
\left(A_{j,h}+\sum_{i\in\mathcal{V}_{j,h}}\,\tfrac{B_{j,h,i}}{z_{j,h,i}}\right)-c_j-T\right)$$
$$+\mu\,\sum_{j=1}^J\,\sum_{h=1}^{H_j}\,x_{j,h}+\nu\,\sum_{j=1}^J\,\sum_{h=1}^{H_j}\,x_{j,h}\,\sum_{i\in\mathcal{V}_{j,h}}\,\alpha_{j,h,i}\,z_{j,h,i}.
$$

Differentiate with respect to $x_{j,h}$ and $z_{j,h,i}$ to get equations for stationary points:
\be\label{dif1}
\tfrac{\partial\,F}{\partial\,x_{j,h}}=-\tfrac{\lambda_j}{x_{j,h}^2}\left(A_{j,h}+\sum_{i\in\mathcal{V}_{j,h}}\,\tfrac{B_{j,h,i}}{z_{j,h,i}}\right)
+\mu+\nu\sum_{i\in\mathcal{V}_{j,h}}\,\alpha_{j,h,i}\,z_{j,h,i}=0
\ee
and
\be\label{dif2}
\tfrac{\partial\,F}{\partial\,z_{j,h,i}}=-\tfrac{\lambda_j\,B_{j,h,i}}{x_{j,h}\,z^2_{j,h,i}}+\nu\,x_{j,h}\,\alpha_{j,h,i}=0.
\ee
From \eqref{dif2} we have
$$
x_{j,h}\,z_{j,h,i}=\sqrt{\tfrac{\lambda_j\,B_{j,h,i}}{\nu\,\alpha_{j,h,i}}}.
$$
Inserting the above expression into \eqref{wiez2} we obtain
$$
\sqrt{\nu}=\tfrac{1}{z}\,\sum_{j=1}^J\,\sqrt{\lambda_j}\,\sum_{h=1}^{H_j}\,\sum_{i\in\mathcal{V}_{j,h}}\,\sqrt{\alpha_{j,h,i}\,B_{j,h,i}}.
$$
On the other hand, inserting it into \eqref{dif1}, upon cancelations, yields
$$
x_{j,h}=\sqrt{\tfrac{\lambda_j\,A_{j,h}}{\mu}}.
$$
Returning now to \eqref{wiez1} we end up with
$$
\sqrt{\mu}=\tfrac{1}{x}\,\sum_{j=1}^J\,\sqrt{\lambda_j}\,\sum_{h=1}^{H_j}\,\sqrt{A_{j,h}}.
$$

Returning to \eqref{tmin} we get
$$
\sqrt{\tfrac{\mu}{\lambda_j}}\,\sum_{h=1}^{H_j}\,\sqrt{A_{h,j}}+\sqrt{\tfrac{\nu}{\lambda_j}}\,\sum_{h=1}^{H_j}\,\sum_{i\in\mathcal{V}_{j,h}}\,
\sqrt{\alpha_{j,h,i}\,B_{j,h,i}}-c_j=T.
$$

Multiply by $v_j:=\sqrt{\lambda_j}$ and plug in the formulas for $\sqrt{\mu}$ and $\sqrt{\nu}$ to arrive at
$$
\tfrac{1}{x}\left(\sum_{k=1}^J\,v_k\,\sum_{g=1}^{H_k}\,\sqrt{A_{k,g}}\right)\,\sum_{h=1}^{H_j}\,\sqrt{A_{h,j}}$$$$
+\tfrac{1}{z}\left(\sum_{k=1}^J\,v_k\,\sum_{g=1}^{H_k}\,\sum_{l\in\mathcal{V}_{k,g}}\,\sqrt{\alpha_{k,g,l}\,B_{k,g,l}}\right)\,
\sum_{h=1}^{H_j}\,\sum_{i\in\mathcal{V}_{j,h}}\,\sqrt{\alpha_{j,h,i}\,B_{j,h,i}}-c_jv_j=Tv_j.
$$
Note that the above equation is equivalent to
$$
D\,\v=T\,\v,
$$
that is $0<T=\lambda$ is the unique positive eigenvalue and $\v$ is the eigenvector related to $T$.

To show that the eigenvector $\v$ attached to the eigenvalue $\lambda$ has all coordinates of the same sign the argument is exactly the same as in the last part of the proof of Theorem \ref{ss}.
\end{proof}

\begin{remark}\label{rema}
{\em In practical applications it is often important that instead of a constant $T$ there are priority weights $\kappa_j>0$, assigned to each constraint \eqref{tmin}, $j=1,\ldots,d$. That is, $T$ is replaced by $\kappa_j T$ at the left hand side of \eqref{tmin} for every $j=1,\ldots,d$. Note that this situation can be rather trivially reduced to the one considered in Theorem \ref{ogol}. This is done by dividing both sides of the $j$th new constraint equation (with $\kappa_j T$ at the left hand side) by $\kappa_j$. Then we obtain \eqref{tmin} with $A_{j,h}$, $B_{j,h,i}$ and $c_j$ changed, respectively, into $A_{j,h}/\kappa_j$, $B_{j,h,i}/\kappa_j$ and $c_j/\kappa_j$. The vectors $\a$ and $\b$ given in \eqref{aa} and $\eqref{bc}$ and the matrix $D$ in \eqref{dmat} have to be updated similarly. Consequently, the solution of the minimization problem as given in \eqref{xjjh} and \eqref{zjjh} refers to the eigenvector $\v$ of such updated matrix $D$; moreover in the formulas \eqref{xjjh} and \eqref{zjjh} all the quantities of the form $A_{j,h}$, $B_{j,h,i}$ have to be changed into $A_{j,h}/\kappa_j$, $B_{j,h,i}/\kappa_j$.}
\end{remark}

Since $D$ is a rank two perturbation of a diagonal matrix one may use Weyl inequalities to establish conditions under which $D$ has a unique positive eigenvalue. Such a sufficient condition is given below.

\begin{proposition}\label{linal}
Let $\a,\,\c\in(0,\infty)^d$ and $\b\in[0,\infty)^d$ be such that
\be\label{warun}
\sum_{i=1}^d\,\tfrac{a_i^2+b_i^2}{c_i}\;-\sum_{\substack{i,j=1\\i\ne j}}^d\,\tfrac{(a_ib_j-a_jb_i)^2}{c_ic_j}>1.
\ee
Then the matrix $D$ defined in \eqref{dmat} has a unique positive eigenvalue $\lambda$.
\end{proposition}

\begin{proof} By $\lambda_1(X)\le \ldots\le \lambda_d(X)$ we denote eigenvalues of $d\times d$ matrix $X$. Take $A=\a\a^T+\b\b^T$ and $B=-\mathrm{diag}(\c)$. Since $A$ is of rank at most 2, we have $\lambda_{d-2}(A)=0$. Consequently, the Weyl inequality \eqref{weyl1} with $k=d-2$ implies $\lambda_{d-2}(D)<0$ since all eigenvalues of $B$ are negative.

Therefore,
\be\label{detd}
\mathrm{sgn}\,\det\,D=(-1)^{d-2}\mathrm{sgn}(\delta_{d-1}\delta_d).
\ee

On the other hand expanding determinant of $D$
$$
\det\,D=\det\,\left[\begin{array}{cccc}
a_1^2+b_1^2-c_1 & a_1a_2+b_1b_2 & \ldots & a_1a_d+b_1b_d \\
a_1a_2+b_1b_2 & a_2^2+b_2^2-c_2 & \ldots & a_2a_d+b_2b_d \\
\ldots & \ldots & \ldots & \ldots \\
a_1a_d+b_1b_d & a_2a_d+b_2b_d & \ldots & a_d^2+b_d^2-c_d
\end{array} \right]
$$
and using the fact that $\a\a^T+\b\b^T$ is of rank at most two we obtain
$$
\det\,D=(-1)^{d-2}\,\sum_{\substack{i,j=1\\i\ne j}}^d\,(a_ib_j-a_jb_i)^2\prod_{\substack{k=1\\k\not\in\{i,j\}}}^d\,c_k+(-1)^{d-1}\,\sum_{i=1}^d\,(a_i^2+b_i^2)\prod_{\substack{k=1\\k\ne i}}^d\,c_k+(-1)^d\prod_{k=1}^d\,c_k$$$$=(-1)^d\prod_{k=1}^d\,c_k\,\left[1-\sum_{i=1}^d\,\tfrac{a_i^2+b_i^2}{c_i}+\sum_{\substack{i,j=1\\i\ne j}}^d\,\tfrac{(a_ib_j-a_jb_i)^2}{c_ic_j}\right].
$$

Comparing the above formula with \eqref{detd} we have
$$
\mathrm{sgn}(\delta_{d-1})\,\mathrm{sgn}(\delta_d)=\mathrm{sgn}\left[1-\sum_{i=1}^d\,\tfrac{a_i^2+b_i^2}{c_i}+\sum_{\substack{i,j=1\\i\ne j}}^d\,\tfrac{(a_ib_j-a_jb_i)^2}{c_ic_j}\right]
$$
which is negative due to \eqref{warun}. Since $\lambda_{d-1}(D)\le \lambda_d(D)$ then necessarily $\lambda_{d-1}(D)<0<\lambda_d(D):=\lambda$.
\end{proof}

Note that if $a_i^2+b_i^2>c_i$, $i=1,\ldots, d$, then $D$ is a matrix with all positive entries and the result of Theorem \ref{linal} holds by the Perron-Frobenius theorem. Under such assumption there are situations in which condition \eqref{warun} may not be satisfied, that is Theorem \ref{linal} does not cover fully this classical case.

\subsection{Application to stratified two-stage sampling}
The population $U$ consists of subpopulation $U_1,\ldots,U_J$. Each subpopulation $U_j$ is partitioned into primary sampling units (PSU's) which are structured into strata $\mathcal{W}_{j,h}$, $h=1,\ldots,H_j$, $j=1,\ldots,J$. Each primary unit $i\in \mathcal{W}_{j,h}$ consists of secondary sampling units (SSU's) which are again structured into strata $\mathcal{W}_{j,h,i,g}$, $g=1,\ldots,G_{j,h,i}$. Let $M_{j,h}$ be the number of PSU's in $\mathcal{W}_{j,h}$ and $N_{j,h,i,g}$ be the number of SSU's in $\mathcal{W}_{j,h,i,g}$.  In the schemes we describe below stratified SRSWOR is used at the second stage.
\subsubsection{SRSWOR at the first stage}
The sample is chosen as follows: At the first stage a sample $\mathcal{S}$ of size $m_{j,h}$ of PSU's is selected from $\mathcal{W}_{j,h}$, $h=1,\ldots,H_j$, $j=1,\ldots,J$,  according to $\mathrm{SRSWOR}$. At the second stage a sample $S$ of size $n_{j,h,i,g}$ of SSU's is selected from $\mathcal{W}_{j,h,i,g}$, $g=1,\ldots,G_{j,h,i}$, only for PSU's $i\in \mathcal{S}$, according to $\mathrm{SRSWOR}$.

The variance of $\pi$-estimator of the total of $\mathcal{Y}$ over subpopulation $U_j$ has the form, see, e.g. S\"arndal, Swensson, Wretman (1992), Ch. 4.3
$$
\sum_{h=1}^{H_j}\,\left(\tfrac{1}{m_{j,h}}-\tfrac{1}{M_{j,h}}\right)M^2_{j,h}D^2_{j,h}+\sum_{h=1}^{H_j}\tfrac{M_{j,h}}{m_{j,h}}\,
\sum_{i\in \mathcal{W}_{j,h}}\,\sum_{g=1}^{G_{j,h,i}}\,\left(\tfrac{1}{n_{j,h,i,g}}-\tfrac{1}{N_{j,h,i,g}}\right)N_{j,h,i,g}^2S_{j,h,i,g}^2
$$
where
$$
D^2_{j,h}=\tfrac{1}{M_{j,h}-1}\sum_{i\in \mathcal{W}_{j,h}}\left(t_i-\bar{t}_{j,h}\,\right)^2
$$
with
$$
t_i=\sum_{k\in V_i}\,y_k\qquad \forall\;\;\mbox{PSU's}\;\;V_i\quad\mbox{and}\quad \bar{t}_{j,h}=\tfrac{1}{M_{j,h}}\sum_{i\in\mathcal{W}_{j,h}}\,t_i
$$
and
$$
S_{j,h,i,g}^2=\tfrac{1}{N_{j,h,i,g}-1}\,\sum_{k\in \mathcal{W}_{j,h,i,g}}\,\left(y_k-\bar{t}_{j,h,i,g}\right)^2
$$
with
$$
\bar{t}_{j,h,i,g}=\tfrac{1}{N_{j,h,i,g}}\,\sum_{k\in \mathcal{W}_{j,h,i,g}}\,y_k.
$$

We assume that the size of the PSU's sample $\mathcal{S}$ is
\be\label{PUS}
\sum_{j=1}^J\,\sum_{h=1}^{H_j}\,m_{j,h}=m
\ee
and expected size of the SSU's sample $S$ is
\be\label{SUS}
\sum_{j=1}^J\,\sum_{h=1}^{H_j}\,\tfrac{m_{j,h}}{M_{j,h}}\,\sum_{i\in\mathcal{W}_{j,h}}\,\sum_{g=1}^{G_{j,h,i}}\,n_{j,h,i,g}=n.
\ee

Additionally we assume that the precision of $\pi$-estimator in every subpopulation is constant, that is
\be\label{copr}
\sum_{h=1}^{H_j}\,\tfrac{M_{j,h}}{m_{j,h}}\left(\gamma_{j,h}+
\sum_{i\in\mathcal{W}_{j,h}}\,\sum_{g=1}^{G_{j,h,i}}\,\tfrac{\beta_{j,h,i,g}}{n_{j,h,i,g}}\right)-c_j=T,\qquad j=1,\ldots,J,
\ee
where
$$
\gamma_{j,h}=\tfrac{1}{T_j^2}\,\left(M_{j,h}D_{j,h}^2-\sum_{i\in\mathcal{W}_{j,h}}\,\sum_{g=1}^{G_{j,h,i}}\,N_{j,h,i,g}\,S_{j,h,i,g}^2\right),
$$
$$
\beta_{j,h,i,g}=\tfrac{1}{T_j^2}\,N_{j,h,i,g}^2\,S_{j,h,i,g}^2\qquad\mbox{and}\qquad c_j=\tfrac{1}{T_j^2}\,\sum_{h=1}^{H_j}\,M_{j,h}\,D_{j,h}^2,
$$
for $T_j=\sum_{h=1}^{H_j}\,\sum_{i\in \mathcal{W}_{j,h}}\,t_i$, $j=1,\ldots,J$.

Now we use Theorem \ref{ogol} with $\mathcal{V}_{j,h}=\bigcup_{i\in\mathcal{W}_{j,h}}\,\widetilde{W}_{j,h,i}$, where  $\widetilde{W}_{j,h,i}=\{\mathcal{W}_{j,h,i,g},\,g=1,\ldots,G_{j,h,i}\}$ is the set of strata of SSU's in the $i$th PSU of $\mathcal{W}_{j,h}$, and with
$$
A_{j,h}=M_{j,h}\,\gamma_{j,h},\quad B_{j,h,(i,g)}=M_{j,h}\,\beta_{j,h,i,g},\quad \alpha_{j,h,(i,g)}=M_{j,h}^{-1},\quad x=m,\quad z=n.
$$
In the above formulas we identified $i$ from Theorem \ref{ogol} with the pair $(i,g)$ in the special setting considered here.
Directly from Theorem \ref{ogol} we get the following result:

\begin{theorem}\label{eig1}
Assume that the matrix $D$ defined by \eqref{dmat} has the unique positive eigenvalue $\lambda$.

Assume that the overall PSU's sample size is $m$ and the expected overall SSU's sample size is $n$. Assume that $\gamma_{j,h}>0$, $h=1,\ldots,H_j$, $j=1,\ldots,J$.

Then the  optimal equal-precision allocation in strata is
$$
m_{j,h}=m\,\frac{v_j\sqrt{M_{j,h}\,\gamma_{j,h}}}{\sum_{k=1}^J\,v_k\sum_{r=1}^{H_k}\,\sqrt{M_{k,h}\,\gamma_{k,h}}},
$$
$h=1,\ldots,H_j,\;j=1,\ldots,J$, and
$$
n_{j,h,i,g}=n \frac{v_j\,M_{j,h}\,\sqrt{\,\beta_{j,h,i,g}}}{m_{j,h}\,\sum_{k=1}^J\,v_k\sum_{r=1}^{H_k}\,\sum_{l\in \mathcal{W}_{k,r}}\,\sum_{s=1}^{G_{k,r,l}}\,\sqrt{\beta_{k,r,l,s}}},
$$
$g=1,\ldots,G_{j,h,i},\;i\in\mathcal{W}_{j,h},\;h=1,\ldots,H_j,\;j=1,\ldots,J$,where $\v=(v_1,\ldots,v_J)^T$ is the unique eigenvector with all coordinates of the same sign of the matrix $D$.

Moreover, the common optimal value of the square of precisions (CV's) $T=\lambda$, the unique positive eigenvalue of $D$.
\end{theorem}

\begin{remark}\label{rema1}
{\em In the case of unequal predesigned precisions in subpopulations described in terms of priority weights $\kappa_j$ assigned to each subpopulation, as pointed out in Remark \ref{rema}, we need to change $\gamma_{j,h}$ into $\gamma_{j,h}/\kappa_j$ and $\beta_{j,h,i,g}$ into $\beta_{j,h,i,g}/\kappa_j$ in Theorem \ref{eig1}. Similarly, $c_j$ changes into $c_j/\kappa_j$ and  the matrix $D$, see \eqref{dmat}, and its eigenvectors and eigenvalues will be automatically updated. The same applies to Th. \ref{eig22}, Th. \ref{eigHR} and Th. \ref{eigHR'} below.}
\end{remark}

\subsubsection{Fixed SSU's sample size within PSU's} To avoid the situation when SSU's sample size is random one can postulate that within PSU's in a given stratum $\mathcal{W}_{j,h}$ it is constant, that is there are numbers (to be found) $n_{j,h}$ denoting SSU's sample size for any $i\in\mathcal{W}_{j,h}$, $j=1,\ldots,J$, $h=1,\ldots,H_j$. Here we assume SRSWOR with no stratification at the second stage. Therefore, while the constraint \eqref{PUS} remains untouched the constraint \eqref{SUS} changes into
\be\label{SUSf}
\sum_{j=1}^J\,\sum_{h=1}^{H_j}\,m_{j,h}n_{j,h}=n.
\ee
The requirement of the common precision yields
$$
T=\tfrac{1}{T_j^2}\,\sum_{h=1}^{H_j}\,\tfrac{M_{j,h}}{m_{j,h}}\left[\left(M_{j,h}D_{j,h}^2-\sum_{i\in\mathcal{W}_{j,h}}\,N_{j,h,i}S_{j,h,i}^2\right)
+\tfrac{1}{n_{j,h}}\sum_{i\in\mathcal{W}_{j,h}}\,N_{j,h,i}^2S_{j,h,i}^2\right]-\tfrac{1}{T_j^2}\,\sum_{h=1}^{H_j}\,M_{j,h}D_{j,h}^2.
$$

Referring again to Theorem \ref{ogol} we take $\#\,\mathcal{V}_{j,h}=1$,
$$
A_{j,h}=M_{j,h}\gamma_{j,h}, \qquad\mbox{where}\qquad \gamma_{j,h}=\tfrac{1}{T_j^2}\left(M_{j,h}D_{j,h}^2-\sum_{i\in\mathcal{W}_{j,h}}\,N_{j,h,i}S_{j,h,i}^2\right),
$$
$$
B_{j,h,1}=M_{j,h}\beta_{j,h},\qquad\mbox{where}\qquad \beta_{j,h}=\tfrac{1}{T_j^2}\sum_{i\in\mathcal{W}_{j,h}}\,N_{j,h,i}^2S_{j,h,i}^2,
$$
$$
c_j=\tfrac{1}{T_j^2}\,\sum_{h=1}^{H_j}\,M_{j,h}D_{j,h}^2,\qquad
\alpha_{j,h,1}=1,\quad x=m,\quad z=n.
$$
Consequently, directly from Theorem \ref{ogol} we have the following result:
\begin{theorem}\label{eig22}
Assume that the matrix $D$ defined by \eqref{dmat} has the unique positive eigenvalue $\lambda$.

Assume that the overall PSU's sample size is $m$ and the expected overall SSU's sample size is $n$. Assume that $\gamma_{j,h}>0$, $h=1,\ldots,H_j$, $j=1,\ldots,J$.

Then the  optimal equal-precision allocation is
$$
m_{j,h}=m\,\frac{v_j\sqrt{M_{j,h}\,\gamma_{j,h}}}{\sum_{k=1}^J\,v_k\sum_{g=1}^{H_k}\,\sqrt{M_{k,g}\,\gamma_{k,g}}},
$$
and
$$
n_{j,h}=n \frac{v_j\,\sqrt{M_{j,h}\beta_{j,h}}}{m_{j,h}\,\sum_{k=1}^J\,v_k\sum_{r=1}^{H_k}\,\sqrt{M_{k,r}\beta_{k,r}}},
$$
$h=1,\ldots,H_j,\;j=1,\ldots,J$,where $\v=(v_1,\ldots,v_J)^T$ is the unique eigenvector with all coordinates of the same sign of the matrix $D$.

Moreover, the common optimal value of precision $T=\lambda$.
\end{theorem}

\subsubsection{Hartley-Rao scheme at the first stage}\label{HRscheme}
The sample is chosen as follows: at the first stage a sample $\mathcal{S}$ of PSU's is selected according to a $\pi$ps sampling introduced by Hartley and Rao (1962), known also as $\pi$ps systematic sampling from randomly ordered list. This sampling procedure is applied in each of the strata of the PSU's. That is an additional auxiliary variable $\mathcal{Z}$ is given in the population of PSU's. Each stratum $\mathcal{W}_{j,h}$ is randomly ordered and then $m_{j,h}$ PSU's are chosen by systematic sampling through $m_{j,h}-1$ jumps of the size one starting at the random point from the interval $[0,1]$. Such a procedure gives a random selection of points $x_1,\ldots, x_{m_{j,h}}\in[0,m_{j,h}]$. We have $$\pi^{(I)}_{j,h,i}=m_{j,h}\tilde{z}_{j,h,i},\qquad\mbox{where}\qquad \tilde{z}_{j,h,i}:=\tfrac{z_i}{z_{j,h}}\quad\mbox{and}\quad z_{j,h}=\sum_{i\in\mathcal{W}_{j,h}}\,z_i$$ for any PSU $i$ from $h$th stratum from $j$th subpopulation. The sample of PSU's in $\mathcal{W}_{j,h}$ consists of such PSU $i$'s that $$\tfrac{z_{j,h}}{m_{j,h}}\,x_k\in \left(\sum_{l=1}^{i-1}\,z_l,\,\sum_{l=1}^i\,z_l\right),\qquad k=1,\ldots,m_{j,h}$$ for some $i=1,\ldots,M_{j,h}$. At the second stage a sample $S$ of SSU's is selected according to $\mathrm{SRSWOR}(n_{j,h,i,g})$ from $\mathcal{W}_{j,h,i,g}$, $g=1,\ldots,G_{j,h,i}$, only for PSU's $i\in \mathcal{S}$.

The approximate variance of $\pi$-estimator of the total of $\mathcal{Y}$ over subpopulation $U_j$ has the form, see Hartley and Rao (1962) (their formula (5.17) for the simplified variance of the $\pi$-estimator for the systematic $\pi$ps sampling and  S\"arndal, Swensson, Wretman (1992), Ch. 4.3, for the variance in two-stage sampling)
$$
\sum_{h=1}^{H_j}\,\tfrac{1}{m_{j,h}}\sum_{i\in\mathcal{W}_{j,h}}\,\omega_{j,h,i}(1+\tilde{z}_{j,h,i})
-\sum_{h=1}^{H_j}\,\sum_{i\in\mathcal{W}_{j,h}}\,\tilde{z}_{j,h,i}\,\omega_{j,h,i}$$$$+\sum_{h=1}^{H_j}\,\tfrac{1}{m_{j,h}}\,
\sum_{i\in \mathcal{W}_{j,h}}\,\tfrac{1}{\tilde{z}_{j,h,i}}\sum_{g=1}^{G_{j,h,i}}\,\left(\tfrac{1}{n_{j,h,i,g}}-\tfrac{1}{N_{j,h,i,g}}\right)N_{j,h,i,g}^2S_{j,h,i,g}^2
$$
where
$$
\omega_{j,h,i}=\tilde{z}_{j,h,i}
\left(\tfrac{y_{j,h,i}}{\tilde{z}_{j,h,i}}-y_{j,h}\right)^2 \quad\mbox{and}\quad y_{j,h}=\sum_{i\in\mathcal{W}_{j,h}}\,y_i
$$
and
$$
S_{j,h,i,g}^2=\tfrac{1}{N_{j,h,i,g}-1}\,\sum_{k\in \mathcal{W}_{j,h,i,g}}\,\left(y_k-\bar{t}_{j,h,i,g}\right)^2
$$
with
$$
\bar{t}_{j,h,i,g}=\tfrac{1}{N_{j,h,i,g}}\,\sum_{k\in \mathcal{W}_{j,h,i,g}}\,y_k.
$$

We assume that the size of the PSU's sample $\mathcal{S}$ satisfies the constraint \eqref{PUS} while the formula \eqref{SUS} for expected size of the SSU's sample assumes the form \be\label{SUSHR}
\sum_{j=1}^J\,\sum_{h=1}^{H_j}\,m_{j,h}\,\sum_{i\in\mathcal{W}_{j,h}}\,\tilde{z}_{j,h,i}\sum_{g=1}^{G_{j,h,i}}\,n_{j,h,i,g}=n.
\ee

Denote $$D^2_{j,h}=\sum_{i\in\mathcal{W}_{j,h}}\,\omega_{j,h,i}(1+\tilde{z}_{j,h,i}).$$

Additionally we assume that the precision of $\pi$-estimator in every subpopulation is constant (actually, since we use an approximate formula for the variances, it is not precision but rather approximate precision), that is
\be\label{copr}
\sum_{h=1}^{H_j}\,\tfrac{1}{m_{j,h}}\left(\gamma_{j,h}+
\sum_{i\in\mathcal{W}_{j,h}}\,\tfrac{1}{\tilde{z}_{j,h,i}}\,\sum_{g=1}^{G_{j,h,i}}\,\tfrac{\beta_{j,h,i,g}}{n_{j,h,i,g}}\right)-c_j=T,\qquad j=1,\ldots,J,
\ee
where
\be\label{gjh}
\gamma_{j,h}=\tfrac{1}{T_j^2}\,\left(D_{j,h}^2-\sum_{i\in\mathcal{W}_{j,h}}\,\tfrac{1}{\tilde{z}_{j,h,i}}\,\sum_{g=1}^{G_{j,h,i}}\,N_{j,h,i,g}\,S_{j,h,i,g}^2\right),
\ee
$$
\beta_{j,h,i,g}=\tfrac{1}{T_j^2}\,N_{j,h,i,g}^2\,S_{j,h,i,g}^2\qquad\mbox{and}\qquad c_j=\tfrac{1}{T_j^2}\,\sum_{h=1}^{H_j}\,\sum_{i\in\mathcal{W}_{j,h}}\,\tilde{z}_{j,h,i}\omega_{j,h,i},
$$
for $T_j=\sum_{h=1}^{H_j}\,\sum_{i\in \mathcal{W}_{j,h}}\,t_i$, $j=1,\ldots,J$.

At this stage we again refer to Theorem \ref{ogol}, once again using the identification $i=(i,g)$. Thus we define
$$
A_{j,h}=\gamma_{j,h},\quad B_{j,h,(i,g)}=\tfrac{\beta_{j,h,i,g}}{\tilde{z}_{j,h,i}},\quad \alpha_{j,h,(i,g)}=\tilde{z}_{j,h,i},\quad x=m,\quad z=n.
$$
As a conclusion from Theorem \ref{ogol} we have the result describing (approximate) optimal equal-precision allocation in the scheme considered in this subsection:

\begin{theorem}\label{eigHR} Assume that the matrix $D$ defined in \eqref{dmat} has a unique positive eigenvalue $\lambda$.
Assume that the overall PSU's sample size is $m$ and the expected overall SSU's sample size is $n$. Assume that $\gamma_{j,h}>0$, $h=1,\ldots,H_j$, $j=1,\ldots,J$.

Then the (approximate) optimal equal-precision allocation in strata is
$$
m_{j,h}=m\,\frac{v_j\sqrt{\gamma_{j,h}}}{\sum_{k=1}^J\,v_k\sum_{g=1}^{H_k}\,\sqrt{\gamma_{k,g}}},
$$
$h=1,\ldots,H_j,\;j=1,\ldots,J$ and
$$
n_{j,h,i,g}=n \frac{v_j\,\sqrt{\beta_{j,h,i,g}}/\tilde{z}_{j,h,i}}{m_{j,h}\,\sum_{k=1}^J\,v_k\sum_{r=1}^{H_k}\,\sum_{l\in \mathcal{W}_{k,r}}\,\sum_{s=1}^{G_{k,r,l}}\,\sqrt{\beta_{k,r,l,s}}},
$$
$g=1,\ldots,G_{j,h,i},\;i\in\mathcal{W}_{j,h},\;h=1,\ldots,H_j,\;j=1,\ldots,J$,where $\v=(v_1,\ldots,v_J)^T$ is the unique eigenvector with all coordinates of the same sign of the matrix $D=\a\a^T+\b\b^T-\mathrm{diag}(\c)$.

Moreover, the common optimal value of precision $T=\lambda$, the unique positive eigenvalue of $D$.
\end{theorem}

\subsubsection{Fixed SSU's sample size within PSU's} Similarly as in the previous section we consider now the situation when samples sizes of SSU's are fixed for PSU's belonging to the same strata within subpopulation. Then the constraint for common (approximate) precision for all subpopulation reads
$$
\sum_{h=1}^{H_j}\,\tfrac{1}{m_{j,h}}\left(\gamma_{j,h}+
\tfrac{1}{n_{j,h}}\,\sum_{i\in\mathcal{W}_{j,h}}\,\tfrac{\beta_{j,h,i}}{\tilde{z}_{j,h,i}}\right)-c_j=T,\qquad j=1,\ldots,J,
$$
where
$$
\gamma_{j,h}=\tfrac{1}{T_j^2}\,\left(D_{j,h}^2-\sum_{i\in\mathcal{W}_{j,h}}\,\tfrac{1}{\tilde{z}_{j,h,i}}\,\sum_{g=1}^{G_{j,h,i}}\,N_{j,h,i,g}\,S_{j,h,i,g}^2\right)
$$
$$
\beta_{j,h,i}=\tfrac{1}{T_j^2}\,N_{j,h,i}^2\,S_{j,h,i}^2\qquad\mbox{and}\qquad c_j=\tfrac{1}{T_j^2}\,\sum_{h=1}^{H_j}\,\sum_{i\in\mathcal{W}_{j,h}}\,\tilde{z}_{j,h,i}\omega_{j,h,i}.
$$
The constraints regarding sizes of samples assume the form
$$
\sum_{j=1}^J\,\sum_{h=1}^{H_j}\,m_{j,h}=m\qquad\mbox{and}\qquad \sum_{j=1}^J\,\sum_{h=1}^{H_j}\,m_{j,h}n_{j,h}=n.
$$

In Theorem \ref{ogol} we take $\#\,\mathcal{V}_{j,h}=1$,
$$
A_{j,h}=\gamma_{j,h}, \qquad
B_{j,h}=B_{j,h,1}=\sum_{i\in\mathcal{W}_{j,h}}\,\tfrac{\beta_{j,h,i}}{\tilde{z}_{j,h,i}},
$$
$c_j$ as above and $\alpha_{j,h,1}=1$, $x=m$, $z=n$.
Consequently, Theorem \ref{ogol} gives the following result:

\begin{theorem}\label{eigHR'} Assume that the matrix $D$ defined in \eqref{dmat} has a unique positive eigenvalue $\lambda$.
Assume that the overall PSU's sample size is $m$ and the overall SSU's sample size is $n$ and the sample SSU's sizes $n_{j,h}$ are fixed (but unknown) within strata in subpopulations. Assume that $\gamma_{j,h}>0$, $h=1,\ldots,H_j$, $j=1,\ldots,J$.

Then the (approximate) optimal equal-precision allocation in strata is
$$
m_{j,h}=m\,\frac{v_j\sqrt{\gamma_{j,h}}}{\sum_{k=1}^J\,v_k\sum_{g=1}^{H_k}\,\sqrt{\gamma_{k,g}}},
$$
$h=1,\ldots,H_j,\;j=1,\ldots,J$ and
$$
n_{j,h}=n \frac{v_j\,\sqrt{B_{j,h}}}{m_{j,h}\,\sum_{k=1}^J\,v_k\sum_{r=1}^{H_k}\,\sqrt{B_{k,r}}},
$$
$h=1,\ldots,H_j,\;j=1,\ldots,J$, where $\v=(v_1,\ldots,v_J)^T$ is the unique eigenvector with all coordinates of the same sign of the matrix $D=\a\a^T+\b\b^T-\mathrm{diag}(\c)$.

Moreover, the common optimal value of precision $T=\lambda$, the unique positive eigenvalue of $D$.
\end{theorem}

\section{Numerical experiments}

In the experiments, described below, we used the method developed in Subsection \ref{HRscheme}, to analyze optimal equal-precision allocation in the Polish Labour Force Survey (LFS).

An artificial population has been created on the basis of results of a sample survey which accompanied the last virtual census in Poland. The sample for this survey was drawn through stratified sampling with strata at a NUTS5 level (we follow the Eurostat standard NUTS nomenclature of territorial units for statistics in EU; in Poland it refers to the level of municipalities). This sample covered 20\% dwellings in the country. To create the population for simulation purposes the records of the sample where cloned together with data related to surveyed persons in each dwelling with the cloning multiplicity equal to rounded corrected weights for dwellings. As a result an artificial population of 13 243 000 households was constructed.

The sampling scheme in the experiment was exactly the same as in the LFS survey, that is; two-stage with stratification at the primary units level. In each such strata of primary units the sample was drawn according to the standard Hartley-Rao method, with first order inclusion probabilities proportional to the number of dwellings in PSU's. The definition of PSU was based on the one adapted in the LFS as well. That is, in urban regions PSU's were identified with so called census clusters and in non-urban areas they were identified with enumeration census areas. The SSU's were just households. The strata definition for PSU's were adapted from the LFS, which resulted in the total of 61 strata of PSU's. Taking into account one of the four subsamples for the quarter of year used in the LFS (one of the two which are new in the survey) the total number of PSU's and SSU's was designed to be $m=1872$ and $n=13676$, respectively.

On the basis of such pseudo-population (with the variable "number of unemployed in the household" transferred from the 20\% survey, which accompanied the last virtual census) suitable initial parameters for the procedure described in Subsection \ref{HRscheme} were prepared. The variables: number of dwellings in PSU's and number of unemployed in a household (SSU) were essential in constructing matrices $\a\,\a^T$, $\b\,\b^T$ and $\mathrm{diag}(\c)$. As subpopulations the NUTS2 level (in Poland it refers to voivodships) was used. The standard {\em R} function {\em eigen} for numerical computation of eigenvalues and eigenvectors was used - see {\em R} Core Team (2013). Examples of $R$-codes we used for optimal equal-precision allocation are available at \href{https://github.com/rwieczor/eigenproblem_sample_allocation}{https://github.com/rwieczor/eigenproblem\_sample\_allocation}. Theoretical value of CV defined through the maximal eigenvalue was numerically found to be about 9.7\%, which is almost exact value of the square root of the largest eigenvalue, as it should be according to the theoretical results obtained in previous sections. The optimal equal-precision allocation, based on the eigenvector related to the largest eigenvalue, was a base for drawing samples of PSU's and then of SSU's. The experiment was repeated independently 100 times with the average PSU's sample size equal to 1872 and the average SSU's sample size equal to 13669. In each experiment, like in the original LFS survey, the precision of estimates of the variable "number of unemployed at NUTS2 levels" was evaluated through a bootstrap method. One of the variations of the bootstrap method was used, where in each stratum a multiple resampling (in this case 500 times) takes place with replacement of $n_h-1$ subsamples out of $n_h$ PSU's selected for the survey in the $h$th stratum - see McCarthy and Snowden (1985) (described also in the monograph Shao and Tu (1995), Ch. 6.2.4). After resampling the original weights for sampling units are properly rescaled and bootstrap variance estimate of the corresponding indicator is obtained by the usual Monte Carlo approximation based on the independent bootstrap replicates. These results were compared to other 100 independent experiments in which the sample was drawn from the pseudo-population according to the standard LFS procedure, which is thoroughly described in Popi\'nski (2006). Actually, we used a simplified version of the standard procedure used in the LFS, because we did not take into account the fact that the real sample consists of four elementary subsamples together with a rotation scheme. Instead, we considered only one of such four elementary subsamples. The same variable was estimated and the precision was evaluated again through the bootstrap procedure. The means of the result over 100 independent experiments are gathered in Table 1 below. One can easily observe that the proposed new procedure gives an average about 14\% gain in CV, when compared to the standard LFS procedure.

\vspace{5mm}
\begin{table}[htbp]
    \centering
    \begin{tabular}{crrrrrr} \hline
         NUTS2  & standard & standard & optimal & optimal & standard & optimal \\ &  SSU allocation &  PSU allocation &  SSU allocation &  PSU allocation &  CV &  CV \\  \hline
       PL11 & 884 & 130 & 888 & 127 & 10.6 & 9.4 \\
        PL12 & 1170 & 156 & 1033 & 153 & 10.8 & 9.5 \\
        PL21 & 884 & 130 & 861 & 130 & 10.9 & 9.4 \\
        PL22 & 1170 & 182 & 999 & 114 & 9.2 & 9.6 \\
        PL31 & 754 & 104 & 813 & 117 & 11.2 & 9.5 \\
        PL32 & 702 & 104 & 661 & 96 & 10.2 & 9.7 \\
        PL33 & 676 & 78 & 731 & 117 & 12.4 & 9.5 \\
        PL34 & 832 & 78 & 882 & 120 & 11.8 & 9.5 \\
        PL41 & 936 & 156 & 927 & 129 & 10.3 & 9.5 \\
        PL42 & 806 & 104 & 836 & 112 & 11.7 & 9.7 \\
        PL43 & 572 & 78 & 814 & 95 & 12.3 & 9.9 \\
        PL51 & 910 & 130 & 880 & 113 & 10.6 & 9.7 \\
        PL52 & 988 & 156 & 942 & 133 & 10.3 & 9.5 \\
        PL61 & 728 & 104 & 779 & 101 & 10.9 & 9.6 \\
        PL62 & 780 & 78 & 754 & 103 & 11.0 & 9.6 \\
        PL63 & 884 & 104 & 869 & 111 & 11.0 & 9.6 \\
        \hline
         Sum   & 13676 & 1872 & 13669 & 1871 &   &   \\ \hline
    \end{tabular}

    \vspace{2mm}
    Table 1: Comparison of allocations and precision between standard and optimal procedures in the LFS on the basis of numerical experiments for a census-based pseudo-population
\end{table}

\section{Conclusions} The method of eigenvalue and eigenvectors was applied to optimal equal-precision allocation in two-stage sampling for the first time in Niemiro and Wesolowski (2001). In the present paper we emphasize its versatility by considering more abstract setting  covering also single-stage sampling (in Section 2) and wider family of two-stage sampling schemes (with stratification at the second stage). In particular, Hartley-Rao sampling at the first stage is taken care of. Additionally the case of constant SSU's sample sizes within PSU's from the same strata is covered by the proposed general methodology. In general the approach is based on looking for a unique positive eigenvalue of a matrix, which is properly defined in terms of population quantities. This matrix appears to be a low-rank ($\le 2$) perturbation of a diagonal matrix. It is proved that the eigenvector associated with the unique positive eigenvalue of this matrix has all components of the same sign. Both the eigenvalue and the eigenvector can be calculated using standard procedures, see e.g. the R-code we posted at \href{https://github.com/rwieczor/eigenproblem_sample_allocation}{https://github.com/rwieczor/eigenproblem\_sample\_allocation}. After the eigenvector is known the optimal equal-precision  allocation is then derived easily. The numerical example shows that application of the proposed method to Polish LFS improves CV of estimates for subpopulations by 14\% on average, when compared to the standard allocation used at present in this survey.

The allocation procedures and formulas developed above, similarly as the classical ones, depend on population quantities as $S^2_{j,h,i,g}$ which, by rule are unknown, and maybe difficult to estimate e.g. from previous surveys. Then a possible approach would be to adopt some model assumptions and replace $S^2_{j,h,i,g}$'s by their model expectations (as done e.g. in Clark (2009) in a somewhat different setting of the problem, when subpopulations may cut across PSU's). An alternative approach would be to refer to auxiliary variable $\mathcal{X}$, which is correlated with the variable of study and available for all PSU's and/or SSU's in the population from administrative registers and using $S^2_{j,h,i,g}(\mathcal{X})$'s instead of $S^2_{j,h,i,g}$'s.

Finally, we mention that the method, we developed in this paper allows us to incorporate different predesigned subpopulations levels of precision priority $\kappa_j>0$, $j=1,\ldots,J$, as described in Section 3.

\vspace{10mm}\small
{\bf References}

\begin{enumerate}

\item {\sc Bethel, J.} Sample allocation in multivariate surveys. {\em Survey Meth.} {\bf 15} (1989), 47-57.

\item {\sc Choudhry, G.H., Rao, J.N.K., Hidiroglou, M.A.} On sample allocation for efficient domain estimation. {\em Survey Meth.} {\bf 38(1)} (2012), 23-29.

\item {\sc Clark, R.G.} Sampling of subpopulations in two stage surveys. {\em Statist. Med.} {\bf 28(29)} (2009), 3697-3717.

\item {\sc Clark, R.G., Steel, D.G.} Optimum allocation of sample to strata and stages with simple additional constraints. {\em J. Royal Statist. Soc. D} {\bf 49} (2000), 197-207.

\item {\sc Cochran, W.G.} {\em Sampling Techniques}, 3rd ed., Wiley, New York, 1977.

\item {\sc Hartley, H.O., Rao, J.N.K.} Sampling with unequal probabilities and without replacement. {\em Ann. Math. Statist.} {\bf 33} (1962), 350-374.

\item {\sc Horn, R., Johnson, C.R.} {\em Matrix Analysis}, Cambridge, 1985.

\item {\sc Huddlestone, H.F., Claypool, P.L., Hocking, R.R.} Optimum allocation to strata using convex programming. {\em Appl. Statist.} {\bf 19} (1970), 273-278.

\item {\sc Kato, T.} {\em A Short Introduction to Perturbation Theory for Linear Operators.} Springer, New York 1981.

\item {\sc Khan, M.G.M., Chand, M.A., Ahmad, N.} Optimum allocation in two-stage and stratified two-stage sampling for multivariate surveys. {\em Proceedings of the Survey Research Methods Section, ASA} (2006), 3215-3220.

\item {\sc Kozak, M.} Method of multivariate sample allocation in agricultural surveys. {\em Biometr. Colloq.} {\bf 34} (2004), 241-250.

\item {\sc Kozak, M., Zieliński, A.} Sample allocation between domains and strata. {\em Int. J. Appl. Math. Statist.} {\bf 3} (2005), 19-40.

\item {\sc Kozak, M., Zieliński, A., Singh, S.} Stratified two-stage sampling in domains: Sample allocation between domains, strata and sampling stages. {\em Statist. Probab. Lett.} {\bf 78} (2008), 970-974.

\item {\sc Lednicki, B., Weso\l owski, J.} Localization of sample between subpopulations. {\em Wiad. Statyst.} {\bf 39(9)} (1994), 2-4 (in Polish).

\item {\sc Lednicki, B., Wieczorkowski, R.} Optimal stratification and sample allocation between subpopulations and strata. {\em Statist. Trans.} {\bf 6(2)} (2003), 287-305.

\item {\sc \L yso\'n, P., Barlik, M., Brochowicz-Lewi\'nska, A., Jacyków, D., Lewandowska, B., Miros\l aw, J., Siwiak, K., W\k agrowska, U.} {\em Household Budget Survey in 2012}. Central Statistical Office, Warszawa 2013.

\item {\sc McCarthy, P.J., Snowden, C.B.} The bootstrap and finite population sampling. In: {\em Vital and Health Statistics}, Ser. 2, no. 95, Public Health Service Publ. 85-1369, US Gov. Print. Office, Washington 1985.

\item {\sc Niemiro, W., Weso\l owski, J.} Fixed precision allocation in two-stage sampling. {\em
Appl. Math.} {\bf 28} (2001), 73-82.

\item {\sc Popi\'nski, W.} Development of the Polish Labour Force Survey. {\em  Statist. Trans.} {\bf 7(5)} (2006), 1009-1030.

\item {\sc R Core Team} {\em R: A language and environment for statistical computing.} R Foundation for Statistical Computing, Vienna 2013. URL http://www.R-project.org

\item {\sc S\"arndal, C.-E., Swensson, B., Wretman, J.} {\em Model Assisted Survey Sampling}, Springer, New York 1992.

\item {\sc Shao, J., Tu, D.} {\em The Jackknife and Bootstrap}, Springer, New York 1995.

\end{enumerate}
\end{document}